\documentclass[12pt]{article}
\usepackage{amsrefs}
\usepackage{amsmath,amssymb,amsthm,}
\usepackage{enumerate,tikz,pxfonts}
\usepackage[all,cmtip]{xy}
\usepackage[applemac]{inputenc}
\usepackage[colorlinks=true, linkcolor=black, citecolor=black]{hyperref}

\def \bs{\backslash}

\def \ol{\overline}

\def \Per{\operatorname{Per}}

\def \R{{\mathbb R}}

\def \SL{\mathrm{SL}}

\def \SO{\operatorname{SO}}

\def \Z{{\mathbb Z}}
\def \({\left(}
\def \){\right)}

\newcommand{\e}
[1]{\emph{#1}\index{#1}}

\newcommand{\tto}
[1]{\stackrel{#1}{\longrightarrow}}

\newcommand{\NN}[3]
{\mbox{\scriptsize$\left(\begin{array}{ccc}1 & #1 & #2 \\ & 1 & #3 \\ &  & 1\end{array}\right)$}}

\newcommand{\PP}[6]
{\mbox{\scriptsize$\left(\begin{array}{ccc}#1 & #4 & #5 \\ & #2 & #6 \\ &  & #3\end{array}\right)$}}

\newtheorem{theorem}{Theorem}[section]

\newtheorem{proposition}[theorem]{Proposition}

\theoremstyle{definition}
\newtheorem{definition}[theorem]{Definition}

\begin{document}

\pagestyle{myheadings} \markright{TRILINEAR FORMS SL(3)}

\title{Invariant trilinear forms for $\SL_3(\R)$}
\author{Anton Deitmar}
\date{}
\maketitle

{\bf Abstract:} We give a detailed analysis of the orbit structure of the third power of the flag variety attached to $\SL_3(\R)$. It turns out that 36 generalized Schubert cells split into 70 orbits plus one continuous family of orbits.

$$ $$

\tableofcontents

\section*{Introduction}

Invariant trilinear forms on flag varieties give intertwining operators in the  category of smooth representations between induced representations and tensor products.
Hence the dimension of these spaces represent intertwining numbers or decomposition numbers of tensor products.
For applications, it is most interesting to consider cases, where these dimensions are finite, which corresponds to rank one situations, see \cites{AZ,BC,BKZ,Clare,CKO,CO,Clerc1,Clerc2,DKS,De1,De2,KSS,MOO,MO}.
In higher rank cases, the space of invariant trilinear forms can be infinite-dimensional.
In particular, in \cite{Koba}, Theorem C, it is shown that if $G$ is algebraic over $\R$ and if there is no open $G$-orbit in the triple product of the flag variety
$$
X=P\bs G\times P\bs G\times P\bs G,
$$
then the space of invariant trilinear forms is infinite-dimensional.
In this paper, we consider the case of the group $\SL_3(\R)$, for which we give a complete analysis of the orbit structure of the corresponding triple product of the flag variety.
It turns out, that the 36 generalized Schubert cells contain 70 isolated orbits and one continuous family of orbits.
In particular, it turns out that there is no open orbit. Hence Theorem C of \cite{Koba} says that there are representations, for which the space of invariant trilinear forms is infinite dimensional.

\section{Orbit structure}
Let $P$ be the minimal parabolic subgroup of $G=\SL_3(\R)$ consisting of all upper triangular matrices.
Then $P$ has Langlands decomposition $P=MAN$, where $A$ is the group of diagonal matrices in $G$ with positive entries, $M$ the group of diagonal matrices in $G$ with entries $\pm 1$, so $M\cong\Z/2\times\Z/2$.
Finally, $N$ is the group of all upper triangular matrices with ones on the diagonal.
Let $K=\SO(3)$. Then $K$ is a maximal compact subgroup of $G$.
We consider the compact manifold
$$
X=(P\bs G)^3=\(M\bs K\)^3.
$$

We write 
$$
D=AM
$$ 
for the group of diagonal matrices in $G$.
Let $N_G(A)$ denote the normalizer of $A$ in $G$. The Weyl group $W=N_G(A)/D$ is isomorphic to the permutation group $\Per(3)$ in three letters.
\begin{definition}
For $w\in W$ we write 
$$
P_v=P\cap v^{-1}Pv,\quad N_v=N\cap v^{-1}Nv.
$$
Since the two parabolic groups $P$ and $v^{-1}Pv$ share the same Levi component $D$, we have
$$
P_v=DN_v,\quad \text{and hence}\quad N\cap P_v=N_v.
$$
\end{definition}

The \e{Bruhat decomposition} is the disjoint decomposition of $G$,
$$
G=\bigsqcup_{w\in W} PwP=\bigsqcup_{w\in W} PwN.
$$
Every $G$-orbit in $X=\(P\bs G\)^3$ contains an element of first coordinate equal to $1$, so we get a bijection
$$
\(P\bs G\times P\bs G\times P\bs G\)/G\tto\cong \{1\}\times\( P\bs G\)^2/P.
$$
Using the Bruhat decomposition in the second and third coordinate, this gives
\begin{align*}
X/G
&\tto\cong 
\bigsqcup_{v\in W}\bigsqcup_{w\in W}\{1\}\times \Big(\(P\bs PvP\)\times P\bs PwP\Big)/P,
\\
&\cong
\bigsqcup_{v\in W}\bigsqcup_{w\in W}\(\{1\}\times \{v\}\times P\bs PwP/(P\cap v^{-1}Pv)\)\\
&\cong
\bigsqcup_{v\in W}\bigsqcup_{w\in W}\{1\}\times \{v\}\times \left[\(N_{w}\bs N/N_v\)\text{ modulo }D-\text{conjugation})\right].
\end{align*}
To justify the last step, consider the map
\begin{align*}
\phi:N&\to P\bs PwP/P_v,\quad n\mapsto PwnP_v
\end{align*}
Since $PwP=PwN$, the map $\phi$ is surjective.
Further,
\begin{align*}
\phi(n')=\phi(n)&\Leftrightarrow
Pwn'P_v=PwnP_v\\
&\Leftrightarrow
wn'=pwndn_v
\tag*{for some $d\in D$, $n_v\in N_v$}\\
&\Leftrightarrow
n'=w^{-1}qw\,n^{d}\,n_v,
\end{align*}
where we have written $n^{d}=d^{-1}nd$.
Now as $n', n^{d} n_v\in N$ we get $w^{-1}qw\in N$ and hence it lies in $N\cap w^{-1}Pw=N_{w}$.
It follows that $\phi$ induces an isomorphism
$N_w\bs N/N_v$ modulo $D$-conjugation to $P\bs PwP/P_v$.

So, defining the \e{Schubert cells}
$$
S_{v,w}=\(P\bs P\cdot 1\times P\bs PvP\times P\bs PwP\)G,
$$
we get a disjoint $G$-stable decomposition into 36 Schubert cells
$$
X=(P\bs G)^3=\bigsqcup_{v\in W}\bigsqcup_{w\in W}S_{v,w}.
$$
The dimension of a cell is
$$
\dim S_{v,w}=3+\dim [v]+\dim [w],
$$
where we have written $[v]=P\bs PvP$.
We write the elements of $W$ as $1,s_1,s_2,z_1,z_2,w_0$ where $w_0$ is the long element and the corresponding Weyl chambers are given as in the following picture
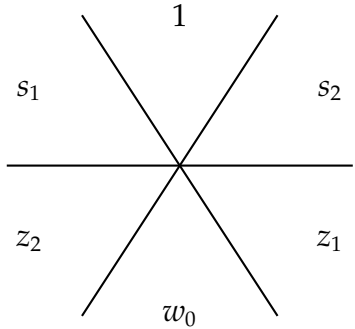
\begin{center}
\begin{tikzpicture}
\draw[thick](-2.3,0)--(2.3,0);
\draw[thick](-1.3,-2)--(1.3,2);
\draw[thick](-1.3,2)--(1.3,-2);
\draw(0,2)node{$1$};
\draw(0,-2)node{$w_0$};
\draw(2,1)node{$s_2$};
\draw(-2,1)node{$s_1$};
\draw(2,-1)node{$z_1$};
\draw(-2,-1)node{$z_2$};
\end{tikzpicture}
\end{center}
Then $s_1$ and $s_2$ generate the Weyl group and $z_1=s_2s_1$, $z_2=s_1s_2$ as well as  $w_0=s_1s_2s_1=s_2s_1s_2$.
For computations, we choose the following representatives in $G$,
\begin{align*}
s_1&=\mbox{\scriptsize$\left(\begin{array}{ccc} & 1 &  \\1 &  &  \\ &  & -1\end{array}\right)$}
&s_2&=\mbox{\scriptsize$\left(\begin{array}{ccc}-1 &  &  \\ &  & 1 \\ & 1 & \end{array}\right)$}\\
z_1&=\mbox{\scriptsize$\left(\begin{array}{ccc} & -1 &  \\ &  & -1 \\1 &  & \end{array}\right)$}
&z_2&=\mbox{\scriptsize$\left(\begin{array}{ccc} &  & 1 \\-1 &  &  \\ & -1 & \end{array}\right)$}\\
w_0&=\mbox{\scriptsize$\left(\begin{array}{ccc} &  & -1 \\ & -1 &  \\-1 &  & \end{array}\right)$}
\end{align*}

The orbit closure inclusion pattern in $P\bs G=\bigsqcup_{w\in W}P\bs PwP$ is given by 
$$
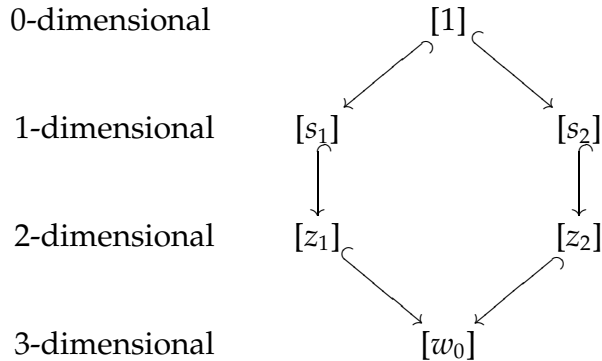
\xymatrix{
\text{0-dimensional}&&[1]\ar@{^(->}[dr]\ar@{^(->}[dl]\\
\text{ 1-dimensional}&[s_1]\ar@{^(->}[d]\ar@{^(->}[drr]&&[s_2]\ar@{_(->}[d]\ar@{_(->}[dll]\\
\text{ 2-dimensional}&[z_1]\ar@{^(->}[dr]&&[z_2]\ar@{^(->}[dl]\\
\text{ 3-dimensional}&&[w_0]&
}
$$
where the arrows indicate containment in the closure, so for instance $[1]\subset \ol{[s_1]}$.
Accordingly, this diagram repeats with the $S_{v,w}$'s, which is to say that
$$
\left.
\begin{array}{c}[v_1]\subset \ol{[v_2]}
\\
\ [w_1]\subset \ol{[w_2]}
\end{array}
\right\}\quad\Leftrightarrow\quad S_{v_1,w_1}\subset\ol{S_{v_2,w_2}}.
$$
So for instance, $S_{s_1,z_1}$ is contained in the closure of $S_{s_1,w_0}$ and in the closure of $S_{z_1,z_1}$.

For $v,w\in W$ let $R_N^{v,w}$ be a set of representatives in $N$ for the equivalence relation
$$
n\sim n'\quad\Leftrightarrow\quad 
n'=n_wdnd^{-1}n_v
$$
for some $d\in D$, $n_v\in N_v$ and $n_w\in N_w$. 
We then get a set $R$ of representatives of $X/G$ of the form
$
R=\big\{ (1,v,wn): v,w\in W,\ n\in R_N^{v,w}\big\}.
$
We can write this suggestively
$
S_{v,w}/G\ \cong\ 1\times v\times w(N/\sim).
$
The stabilizer of $(1,v,wn)$ equals
\begin{align*}
G_{(1,v,wn)}&=P\cap v^{-1}Pv\cap (wn)^{-1}Pwn\\
&= D N_v\cap n^{-1}\(w^{-1}Pw\)n.
\end{align*}

We clearly have $N_1=N$.
A computation shows that 
\begin{align*}
s_1^{-1}Ps_1&=\left(\begin{array}{ccc}* & 0  & * \\ * & * & * \\ 0&  0 & *\end{array}\right)
&s_2^{-1}Ps_2&=\left(\begin{array}{ccc}* & * & * \\ 0& * & 0 \\ 0& * & *\end{array}\right)\\
z_1^{-1}Pz_1&=\left(\begin{array}{ccc}* & 0 & 0 \\   * & * & * \\  * & 0 & *\end{array}\right)
&z_2^{-1}Pz_2&=\left(\begin{array}{ccc}* & * & 0 \\ 0 & * & 0 \\ * & * & *\end{array}\right).
\end{align*}
This implies
\begin{align*}
N_{s_1}&=\NN 0**
&N_{s_2}&=\NN **0\\
N_{z_1}&=\NN00*
&N_{z_2}&=\NN*00.
\end{align*}

We now classify all orbits by the Schubert cells.
The dimensions of the Schubert cells range from 3 to 9. The orbit dimensions cannot exceed $8=\dim G$, therefore the open cell $S_{w_0,w_0}$ must contain a continuous family of orbits. We introduce the notation
$$
n(x,y,z)=\NN xyz, \qquad d(a,b,c)=\PP abc\ \ \ .
$$
Note that
$$
n(x,y,z)^{-1}=n(-x,xz-y,-z).
$$

\begin{proposition}
The open cell $S_{w_0,w_0}$ contains a family of orbits of maximal dimension, which is 8, parametrized by $u\in\R^\times$ and given by $(1,w_0,w_0n)$ with
$n=n(1,1,u)$. 

There are $7$ more orbits: 3 orbits of maximal dimension in $S_{w_0,w_0}$ given by $n(0,1,1)$, $n(1,0,1)$ and $n(1,1,0)$, in each case the stabilizer is trivial.
There are 3 orbits of dimension 7, which we list by a representative and the corresponding stabilizer.
$$
\begin{array}{c|c}
n&\text{stabilizer of }(1,w_0,w_0n)\\
\hline
n(1,0,0)& \big\{d(a,a,1/a^2):a\in\R^\times\big\}\\
n(0,1,0)& \big\{d(a,1/a^2,a):a\in\R^\times\big\}\\
n(0,0,1)& \big\{d(1/a^2,a,a):a\in\R^\times\big\}.
\end{array}
$$
Finally, there is one orbit of dimension 6 given by $n(0,0,0)$ with stabilizer $D$.
\end{proposition}

\begin{proof}
The given elements form a set of representatives of $N$ modulo $D$-conjugation, so they parametrize the orbits in the cell.
The inverse of $n(1,1,u)$ is $n(-1,u-1,-u)$ and one notes that for given $d\in D$ in order to have $n(1,1,u)\,d\,n(-1,u-1,-u)\in D$, one must have $d=1$.
This implies the triviality of the stabilizer. The other cases are treated similarly.
\end{proof}

\begin{proposition}
The cell $S_{w_0,z_1}$ contains 4 orbits which are listed by representative $(1,w_0,z_1n)$, dimension of orbit, and stabilizer.
$$
\begin{array}{c|c|c}
n&\text{dimension}&\text{stabilizer}\\
\hline
n(1,1,0)&8&\big\{1\big\}
\\
n(1,0,0)&7&\big\{d(a,a,1/a^2):a\in\R^\times\big\}
\\
n(0,1,0)&7&\big\{d(a,1/a^2,a):a\in\R^\times\big\}
\\
n(0,0,0)&6&D
\end{array}
$$
\end{proposition}

\begin{proof}
Similar to the last proposition.
\end{proof}

\begin{proposition}
The cell $S_{w_0,z_2}$ contains 4 orbits which are listed by representative $(1,w_0,z_2n)$, dimension of orbit, and stabilizer.
$$
\begin{array}{c|c|c}
n&\text{dimension}&\text{stabilizer}\\
\hline
n(0,1,1)&8&\big\{1\big\}
\\
n(0,0,1)&7&\big\{d(1/a^2,a,a):a\in\R^\times\big\}
\\
n(0,1,0)&7&\big\{d(a,1/a^2,a):a\in\R^\times\big\}
\\
n(0,0,0)&6&D
\end{array}
$$
\end{proposition}

\begin{proposition}
\begin{itemize}
\item The cell $S_{w_0,s_1}$ contains 2 orbits
$$
\begin{array}{c|c|c}
n&\text{dimension}&\text{stabilizer}\\
\hline
n(1,0,0)&7&\big\{d(a,a,1/a^2):a\in\R^\times\big\}
\\
n(0,0,0)&6&D
\end{array}
$$
\item The cell $S_{w_0,s_2}$ contains 2 orbits
$$
\begin{array}{c|c|c}
n&\text{dimension}&\text{stabilizer}\\
\hline
n(0,0,1)&7&\big\{d(1/a^2,a,a):a\in\R^\times\big\}
\\
n(0,0,0)&6&D
\end{array}
$$
\item The cell $S_{w_0,1}$ equals one orbit given by $(1,w_0,1)$, the dimension is 6 and the stabilizer is $D$.
\end{itemize}
\end{proposition}

\begin{proposition}
\begin{itemize}
\item The cell $S_{z_1,z_1}$ contains 3 orbits
$$
\begin{array}{c|c|c}
n&\text{dimension}&\text{stabilizer}\\
\hline
n(1,0,0)&7&\big\{ d(a,a,1/a^2): a\in\R^\times\big\}
\\
n(0,1,0)&6&\left\{
\PP a{1/a^2}a\ \ x: a\in\R^\times;\ x\in\R\right\}
\\
n(0,0,0)&5&D\NN \ \ * 
\end{array}
$$
\item The cell $S_{z_1,z_2}$ contains 2 orbits
$$
\begin{array}{c|c|c}
n&\text{dimension}&\text{stabilizer}\\
\hline
n(0,1,0)&7&\big\{ d(a,1/a^2,a):a\in\R^\times\big\}
\\
n(0,0,0)&6&D\ \ 
\end{array}
$$
\item The cell $S_{z_1,s_1}$ contains 2 orbits
$$
\begin{array}{c|c|c}
n&\text{dimension}&\text{stabilizer}\\
\hline
n(1,0,0)&6&\left\{
\PP aa{1/a^2}\ \ x: a\in\R^\times;\ x\in\R\right\}
\\
n(0,0,0)&5&D\NN \ \ *
\end{array}
$$
\item
The cell $S_{z_1,s_2}$ is one orbit and we have
$$
\begin{array}{c|c|c}
n&\text{dimension}&\text{stabilizer}\\
\hline
n(0,0,0)&6&D 
\end{array}
$$
\item The cell $S_{z_1,1}$ is one orbit with
$$
\begin{array}{c|c|c}
n&\text{dimension}&\text{stabilizer}\\
\hline
n(0,0,0)&5&D\NN \ \ * 
\end{array}
$$
\end{itemize}
\end{proposition}

\begin{proposition}
\begin{itemize}
\item The cell $S_{z_2,z_2}$ contains 3 orbits
$$
\begin{array}{c|c|c}
n&\text{dimension}&\text{stabilizer}\\
\hline
n(0,0,1)&7&\big\{
d(1/a^2,a,a) : a\in\R^\times\big\}
\\
n(0,1,0)&6&\left\{
\PP a{1/a^2}az\ \  : a\in\R^\times;\ z\in\R\right\}
\\
n(0,0,0)&5&D\NN *\ \ 
\end{array}
$$
\item The cells $S_{z_2,s_1}$ and $S_{z_2,1}$ both are one orbit each. For $S_{z_2,s_1}$ we have
$$
\begin{array}{c|c|c}
n&\text{dimension}&\text{stabilizer}\\
\hline
n(0,0,0)&6&D 
\end{array}
$$
and for $S_{z_2,1}$ it is
$$
\begin{array}{c|c|c}
n&\text{dimension}&\text{stabilizer}\\
\hline
n(0,0,0)&5&D \NN *\ \ 
\end{array}
$$
\item The cell $S_{z_2,s_2}$ contains 2 orbits
$$
\begin{array}{c|c|c}
n&\text{dimension}&\text{stabilizer}\\
\hline
n(0,0,1)&6&\left\{
\PP {1/a^2}aax\ \ :a\in\R^\times;\ x,y\in\R\right\}
\\
n(0,0,0)&5&D\NN *\ \ 
\end{array}
$$
\end{itemize}
\end{proposition}

\begin{proposition}
\begin{itemize}
\item
The cell $S_{s_1,s_1}$ contains 2 orbits
$$
\begin{array}{c|c|c}
n&\text{dimension}&\text{stabilizer}\\
\hline
n(1,0,0)&5&\left\{
\PP aa{1/a^2}\ yz : a\in\R^\times;\ x,y\in\R\right\}\\
n(0,0,0)&4&\left\{
\PP ab{1/ab}\ yz : a,b\in\R^\times;\ x,y\in\R\right\} 
\end{array}
$$
\item
The cell $S_{s_1,s_2}$ is one orbit
$$
\begin{array}{c|c|c}
n&\text{dimension}&\text{stabilizer}\\
\hline
n(0,0,0)&5&\left\{
\PP ab{1/ab}\ y\ : a,b\in\R^\times;\ y\in\R\right\}
\end{array}
$$
\item
The cell $S_{s_1,1}$ is one orbit
$$
\begin{array}{c|c|c}
n&\text{dimension}&\text{stabilizer}\\
\hline
n(0,0,0)&4&\left\{
\PP ab{1/ab}\ yz : a,b\in\R^\times;\ y,z\in\R\right\}
\end{array}
$$
\item
The cell $S_{s_2,s_2}$ contains 2 orbits
$$
\begin{array}{c|c|c}
n&\text{dimension}&\text{stabilizer}\\
\hline
n(0,0,1)&5&\left\{
\PP {1/a^2}aaxy\ : a\in\R^\times;\ x,y\in\R\right\}\\
n(0,0,0)&4&\left\{
\PP ab{1/ab}xy\  : a,b\in\R^\times;\ x,y\in\R\right\}
\end{array}
$$
\item
The cell $S_{s_2,1}$ is one orbit
$$
\begin{array}{c|c|c}
n&\text{dimension}&\text{stabilizer}\\
\hline
n(0,0,0)&4&\left\{
\PP ab{1/ab}xy\ : a,b\in\R^\times;\ x,y\in\R\right\}
\end{array}
$$
\item The cell $S_{1,1}$ is one orbit of dimension 3. The stabilizer is $P$.
\end{itemize}
\end{proposition}

\begin{proposition}
In $X$, there is one family of orbits parametrized by $u\in\R^\times$ and 70 more orbits.
These are distributed over the Schubert cells as in the first of the following tables.
The second table gives the dimensions of the orbits in each cell. For instance, $6,7^2,8$ stands for one cell of dimension 6, two of dimension 7 and one of dimension 8.
$$
\begin{array}{c|c|c|c|c|c|c}
&1&s_1&s_2&z_1&z_2&w_0\\
\hline
1&1&1&1&1&1&1\\
\hline
s_1&1&2&1&2&1&2\\
\hline
s_2&1&1&2&1&2&2\\
\hline
z_1&1&2&1&3&2&4\\
\hline
z_2&1&1&2&2&3&4\\
\hline
w_0&1&2&2&4&4&7
\end{array}
\qquad
\begin{array}{c|c|c|c|c|c|c}
&1&s_1&s_2&z_1&z_2&w_0\\
\hline
1&3&4&4&5&5&6\\
\hline
s_1&4&4, 5&5&5, 6&6&6,7\\
\hline
s_2&4&5&4,5&6&5,6&6,7\\
\hline
z_1&5&5,6&6&5,6,7&6,7&6,7^2,8\\
\hline
z_2&5&6&5,6&6,7&5,6,7&6,7^2,8\\
\hline
w_0&6&6,7&6,7&6,7^2,8&6,7^2,8&6,7^3,8^3
\end{array}
$$
\end{proposition}

\begin{proof}
This follows from the previous propositions together with the observation that the  orbits structure of $S_{v,w}$ is the same as that of $S_{w,v}$ because of the flip $X\to X$, $(x,y,z)\mapsto (x,z,y)$.
\end{proof}

{\small Mathematisches Institut\\
Auf der Morgenstelle 10\\
72076 T\"ubingen\\
Germany\\
\tt deitmar@uni-tuebingen.de}

\today

\end{document}